\documentclass[11pt]{article} %{bxjsarticle}

\usepackage{amssymb}
\usepackage{amsthm}
\usepackage{amsmath}
\usepackage[dvipdfmx]{graphicx}

\usepackage{arydshln}
\usepackage{calc}
\usepackage{enumitem}
\usepackage{hyperref}
\usepackage{optidef}
\usepackage{aliascnt}

\makeatletter
\@tempcnta\z@
\loop\ifnum\@tempcnta<26
\advance\@tempcnta\@ne
\expandafter\edef\csname f\@Alph\@tempcnta\endcsname{\noexpand\mathfrak{\@Alph\@tempcnta}}
\expandafter\edef\csname b\@Alph\@tempcnta\endcsname{\noexpand\mathbb{\@Alph\@tempcnta}}
\expandafter\edef\csname c\@Alph\@tempcnta\endcsname{\noexpand\mathcal{\@Alph\@tempcnta}}
\repeat
\makeatother

\newcommand{\trans}[1]{{#1}^\mathrm{T}}

\newcommand{\innerp}[2]{{#1} \bullet {#2}}

\newcommand{\submatrix}[3]{{#1}_{{#2}{#3}}}
\newcommand{\idiagonal}[1]{\mathrm{d}(#1)}
\newcommand{\ioffdiagonal}[1]{\mathrm{od}(#1)}

\DeclareMathOperator{\trace}{tr}
\DeclareMathOperator{\rank}{rank}

\newcommand{\newaliastheorem}[3]{%
    \newaliascnt{#2}{#1}%
    \newtheorem{#2}[#2]{#3}%
    \aliascntresetthe{#2}%
    \expandafter\providecommand\csname #2autorefname\endcsname{#3}%
}

\newtheorem{theorem}{Theorem}[section]

\newaliastheorem{theorem}{define}{Definition}
\newaliastheorem{theorem}{lemma}{Lemma}
\newaliastheorem{theorem}{corollary}{Corollary}
\newaliastheorem{theorem}{proposition}{Proposition}
\newaliastheorem{theorem}{assumption}{Assumption}

\usepackage{mathtools} 
\usepackage{comment}
\usepackage{fancybox,color}
\definecolor{lred}{rgb}{1,0.8,0.5}
\definecolor{lblue}{rgb}{0.8,0.8,1}
\definecolor{dred}{rgb}{0.6,0,0}
\definecolor{dblue}{rgb}{0,0,0.7}
\definecolor{violet}{rgb}{0.5804,0.0000,0.8275}
\definecolor{purple}{rgb}{0.2400,0.5700,0.2500}
\definecolor{TGreen}{rgb}{0,0.50,0.10}
\def\@themcountersep{}
\def\0{\mbox{\bf 0}}
\def\1{\mbox{\bf 1}}
\def\2{\mbox{\bf 2}}
\def\3{\mbox{\bf 3}}
\def\4{\mbox{\bf 4}}
\def\5{\mbox{\bf 5}}
\def\6{\mbox{\bf 6}}
\def\7{\mbox{\bf 7}}
\def\8{\mbox{\bf 8}}
\def\9{\mbox{\bf 9}}

\newdimen\zhige \zhige=0pt
\def\chige#1{{\setbox\zhige\hbox{#1}\ifdim\ht\zhige=1ex\accent24 #1%
  \else\ooalign{\unhbox\zhige\crcr\hidewidth\char24\hidewidth}\fi}}

\def\e{\mbox{\boldmath $e$}}

\def\DC{\mbox{$\cal D$}}
\def\EC{\mbox{$\cal E$}}
\def\FC{\mbox{$\cal F$}}

\def\VC{\mbox{$\cal V$}}

\def\Real{\mbox{$\mathbb{R}$}}

\def\SMAT{\mbox{$\mathbb{S}$}}

\usepackage[normalem]{ulem}
\ifx11 % 11 => use these commands, 10 => NOT use these commands

\fi

\title{Exact SDP relaxations of  quadratically constrained quadratic programs with forest structures }

\author{ 
Godai Azuma\thanks{Department of  Mathematical and Computing Science,
        Tokyo Institute of Technology, 2-12-1-W8-29 Oh-Okayama, Meguro-ku, Tokyo 152-8552, Japan. }
\and 
\normalsize
        Mituhiro Fukuda\thanks{Department of  Mathematical and Computing Science,
        Tokyo Institute of Technology, 2-12-1-W8-41 Oh-Okayama, Meguro-ku, Tokyo 152-8552, Japan
         ({\tt mituhiro@is.titech.ac.jp}). 
	} 
\and 
\normalsize 
	Sunyoung Kim\thanks{Department of Mathematics, Ewha W. University, 52 Ewhayeodae-gil, Sudaemoon-gu, 
	Seoul 	03760, Korea  ({\tt skim@ewha.ac.kr}). The research was supported
        by  NRF 2017-R1A2B2005119.}
 \and 	
\normalsize
        Makoto Yamashita\thanks{Department of  Mathematical and Computing Science,
            Tokyo Institute of Technology, 2-12-1-W8-29 Oh-Okayama, Meguro-ku, Tokyo 152-8552, Japan        
             ({\tt Makoto.Yamashita@c.titech.ac.jp}).
	This research was partially supported by JSPS KAKENHI (Grant number: 20H04145).
 	}	
}
\date{September, 2020}

\begin{document}

\maketitle

\begin{abstract} \noindent
    We study the exactness of the semidefinite programming (SDP) relaxation of quadratically constrained quadratic programs (QCQPs).
   With the aggregate sparsity matrix  from the data matrices of a QCQP with $n$ variables, 
   the rank and positive semidefiniteness of the matrix are examined.
   We prove that  if the rank of the aggregate sparsity matrix is not less than $n-1$ and the matrix remains positive semidefinite
    after replacing  some off-diagonal nonzero elements  with zeros, then the standard SDP relaxation provides an exact optimal solution for the QCQP under feasibility assumptions. 
    In particular, we demonstrate that QCQPs with forest-structured aggregate sparsity matrix,  such as the tridiagonal or arrow-type matrix,
    satisfy the exactness condition on the rank.
     The exactness is attained by considering the feasibility of the dual SDP relaxation, the strong duality of SDPs, and a sequence of QCQPs with perturbed objective functions, under the assumption that the feasible region is compact. We generalize our result for a wider class of QCQPs by applying simultaneous tridiagonalization on the data matrices. Moreover,  simultaneous tridiagonalization is applied to a matrix pencil so that QCQPs with two constraints can be solved exactly by the SDP relaxation.
\end{abstract}

\vspace{0.5cm}

\noindent
{\bf Key words. } Quadratically constrained quadratic programs, exact semidefinite relaxations,
forest graph, the rank of aggregated sparsity matrix.

\vspace{0.5cm}

\noindent
{\bf AMS Classification. } 
90C20,  	%Quadratic programming
90C22,  	%Semidefinite programming
90C25, 	%Convex programming
90C26.  	%Nonconvex programming, global optimization

%\input sect1.tex

%!TEX root = ./main.tex
\section{Introduction} 

We consider a quadratically constrained quadratic program (QCQP):
\begin{mini}
    {}{\trans{x}Q^0x + 2\trans{q_0}x}{\label{eq:qcqp}}{}
    \addConstraint{\trans{x}Q^px + 2\trans{q_p}x}{\leq b_p, \quad}{p = 1, \ldots, m,}
\end{mini}
where $Q^p \in \bS^n$, $q^p \in \bR^n$,
and $b_p \in \bR$ for $p = 0, 1, \ldots, m$ are problem data and $x \in \Real^n$ is  the variable.
Nonconvex QCQPs of the form \eqref{eq:qcqp} arise from a wide range of applications, for instance,
sensor network localization problems~\cite{Biswas2006},
quadratic assignment problems~\cite{Koopmans1957, Povh2009},
equally deployment problems~\cite{safarina2020conic, safarina2019polyhedral},
and optimal power flow problems~\cite{BOSE2015,Lavaei2012, Zhou2019}.
 As  some NP-hard problems can be reformulated by QCQPs, 
they are known to be NP-hard  in general.
%because some NP-hard problems can be reformulated by QCQPs.
Nonconvex QCQPs are approximately solved  with tractable convex relaxations such as 
semidefinite programming (SDP) relaxations.

SDP relaxations of QCQPs are regarded as a powerful convex relaxation that provides
tight lower bounds for \eqref{eq:qcqp} 
\cite{NESTROV2000}.
Both   theoretical and computational aspects of 
 SDP relaxations  \cite{ANJOS2012,Bao2011, BOSE2015,Burer2019,kim2003exact,NESTROV2000,Wang2019,Wang2020,WOLKOWICZ2000,Zhou2019}
have been extensively studied.  
 For  QCQPs in the form of \eqref{eq:qcqp},   by letting $X = x\trans{x}$
and relaxing it to $X - x\trans{x} \succeq O$,
we have the standard SDP relaxation    as 
\begin{mini}
    {}{\innerp{Q^0}{X} + 2\trans{q_0}x}{\label{eq:sdr}}{}
    \addConstraint{\innerp{Q^p}{X} + 2\trans{q_p}x}{\leq b_p, \quad}{p = 1, \ldots, m}
    \addConstraint{  X}{\succeq x\trans{x},}
\end{mini}
where $\innerp{Q^p}{X}$ denotes the Frobenius inner product of $Q^p$ with $X$,
and $X \succeq x\trans{x}$ denotes that $X - x\trans{x}$ is positive semidefinite.
Computational studies on the SDP relaxation \eqref{eq:sdr} for an approximate lower bound of \eqref{eq:qcqp}
have been focused on improving the computational efficiency of solution methods.
Primal-dual interior-point methods or bundle's methods are some of widely used computational methods to solve
large-sized SDP relaxations \cite{WOLKOWICZ2000}.
 In particular, the aggregate sparsity of
  data matrices has been successfully used to reduce the size of the SDP relaxation 
when imploying primal-dual interior-point methods \cite{fukuda2001exploiting, kim2011exploiting,nakata2003exploiting}.
Recently, more efficient algorithms based on the first-order methods, for instance, SDPNAL+ \cite{YST2015} 
and  BBCPOP \cite{ITO2018},
 have been introduced for large-scale QCQPs.

For the theoretical study on the SDP relaxation, the rank of the SDP solution 
plays an important role. As the feasible set of the SDP relaxation is  larger  than that of the original QCQP in general,
an approximate solution to nonconvex QCQP \eqref{eq:qcqp} is usually obtained by solving the SDP relaxation \eqref{eq:sdr}.
 The rank of the SDP solution can be determined after or prior to  solving \eqref{eq:sdr}. 
If the rank of the  computed SDP solution is one, or the matrix $[1, \trans{(x^*)}; x^*, X^*]$ is rank-1, then the SDP relaxation is called exact.
With the rank-1 SDP solution, $x^*$ satisfying $X^* = x^*\trans{x^*}$ recovers the relaxed constraint. 
For some class of QCQPs, however, the rank of the SDP solution is known prior to solving \eqref{eq:sdr}.
QCQPs with nonpositive off-diagonal data matrices were known to be solved exactly by the SDP relaxation in  \cite{BOSE2015,kim2003exact,Lavaei2012}.
In particular, the
exactness of the SDP relaxation for QCQPs with complex variables associated with
connected and acyclic graphs was studied in
\cite{BOSE2015} where 
some sign properties of $Q^p$ $(p=1,\ldots,m)$ were assumed.
Low rank SDP solutions and  the upper bounds  for the rank of the SDP solution were also studied by
Pataki~\cite{Pataki1998},  
Laurent and Vavitsiotis~\cite{Laurent2014}, 
and Madani et al. \cite{Madani2017}.

Recently,
Burer and Ye  in \cite{Burer2019} proposed a method  to determine the rank of SDP solutions for some class of QCQPs 
prior to solving the SDP relaxation.  
For diagonal QCQPs with diagonal $Q^0, Q^1, \ldots, Q^m$, they showed that 
the rank of  SDP solutions  is bounded above by $n - f + 1$,
where the feasibility number $f$ is determined by considering  the systems for $j=1,\ldots,n$:
\begin{align*}
  &  \innerp{Q^0}{X} + [q_0]_jx_j = -1, \\
   & \innerp{Q^i}{X} + [q_i]_jx_j \leq 0,\quad \forall i = 1, \ldots, m, \\
   & \text{$X$: diagonal},\quad X_{kk} \geq 0,\quad \forall k \neq j.
\end{align*}
More precisely,
$f= | \left\{j  \,\middle|\, \text{the above system with $j$ is feasible, $1 \leq j \leq n$ }\right\}|$.
Some exactness conditions for the diagonal QCQPs were provided
by analyzing the case where the upper bound $n-f+1$ equals one, and  their result was extended to random or
 non-random QCQPs. 
More recently, Wang and Kilin\chige{c}-Karzan~\cite{Wang2020} 
analyzed the faces of the convex Lagrangian multipliers $\Gamma$
of the SDP relaxation for a QCQP. They stablished conditions for
which the exact SDP relaxation holds and, in particular, it includes the
result of \cite{Burer2019} for diagonal QCQPs.

Special classes of QCQPs that
 admit the exact SDP relaxation has also been studied. For instance,
the  Generalized Trust-Region Subproblem (GTRS)
that minimizes a quadratic objective over a quadratic constraint is such  a class.
The GTRS is, in fact, a QCQP \eqref{eq:qcqp} with only one constraint ($m = 1$).
It generalizes the classical Trust-Region Subproblem (TRS) in which
a quadratic objective is minimized over a Euclidean ball.
Since the objective of the TRS is allowed to be nonconvex,
the TRS is  nonlinear and nonconvex;
however,  its SDP relaxation is always exact.
The GTRS shares nice properties with the TRS.
In fact, under the Slater's condition
due to the  S-lemma~\cite{Polik2007},
the GTRS admits an exact SDP relaxation.
Generalized eigenvalue problems are closely related to the GTRS.
Adachi and Nakatsukasa in ~\cite{Adachi2019} developed an eigenvalue-based algorithm for the GTRS.
Recently, Wang and Kilin\chige{c}-Karzan~\cite{Wang2019} analyzed
the convex hull of a nonconvex feasible set
using the generalized eigenvalue of a matrix pencil $Q^0 - \lambda Q^1$
which is also used in \autoref{ssec:gtrs} of this paper. 

The main purpose of  this paper is to present
 sufficient conditions for the 
 SDP relaxation to be exact for some  classes of QCQPs,  considering the aggregate sparsity matrix of  data matrices $Q^0, Q^1, \ldots, Q^m$.
We assume that the aggregate sparsity matrix is positive semidefinite
in addition to mild feasibility assumptions.
We show that
if (i)  the rank of the aggregate sparsity matrix is not less than $n - 1$ for any nonzero values of the matrix and
(ii) the positive semidefiniteness of the matrix is maintained even when some of off-diagonal elements of the matrix become 
zeros; then
 the SDP relaxation of the  QCQPs is exact.  
The aggregate sparsity matrices satisfying (i) and (ii) focused in this paper are the matrices associated to forest-structured graphs 
such as tridiagonal and arrow-type matrix.
We call QCQPs \eqref{eq:qcqp}   forest-structured QCQPs, tridiagonal QCQPs, or arrow-type QCQPs if the indices of  maximal cliques obtained from the aggregate sparsity matrix form a ``forest'', a tridiagonal matrix, or an arrow-type matrix,  respectively, sometimes with permutation.
These classes of QCQPs admit the exact SDP relaxation under some assumptions on the feasible set of  QCQPs.
We also extend our results on tridiagonal QCQPs to general QCQPs via simultaneous tridiagonalization.

The aggregate sparsity matrix from the data matrices of QCQPs in this paper
  is employed to examine the rank of the SDP dual solution, while
 it  has been mostly studied for improving computational efficiency of solving the SDP relaxation \cite{fukuda2001exploiting,kim2011exploiting,nakata2003exploiting,SHEEN2019}.
Moreover, our results show that 
 the exact SDP relaxation can be proved for forest-structured QCQPs, regardless of  signs of  data matrices \cite{BOSE2015,kim2003exact,SOJOUDI2014exact}. 
Note that Bose et al.~\cite{BOSE2015} also considered connected 
and acyclic graphs associated with the aggregate sparsity matrices.  For the classes of
forest-structured QCQPs, including tridiagonal or arrow-type QCQPs, the second-order cone relaxation also provides the exact 
optimal solution \cite{kim2003exact}.
We also note that the exactness conditions  by Burer and Ye~\cite{Burer2019}
cannot be used for determining the exactness of the SDP relaxation for tridiagonal QCQPs
since  diagonal QCQPs  in \cite{Burer2019} is  a  special case of  tridiagonal QCQPs.

For tridiagonal QCQPs, we consider  at most $n-1$ 
systems corresponding to the elements of the positive semidefinite variable matrix in the dual SDP relaxation, which is described as \eqref{eq:system} in \autoref{sec:exactness_of_triqcqp}.
Each system consists of  constraints in the dual SDP relaxation of \eqref{eq:sdr} with $q_p=0$ for $p=0, 1,\ldots,m$ and
 the constraint that an
element of the positive semidefinite variable matrix should be zero.
If the system has no solutions
for all $(k, \ell) \in \{(i, i + 1) \,|\, i = 1, \ldots, n - 1\}$, then we show that the SDP relaxation is exact. 
More precisely, 
our condition on the exactness is that the system \eqref{eq:system} is not feasible for all $(k, \ell) \in \{(i, i + 1) \,|\, i = 1, \ldots, n - 1\}$.
Since the number of elements in $\{(i, i + 1) \,|\, i = 1, \ldots, n - 1\}$ is $n - 1$,
we can determine whether the SDP relaxation of a given tridiagonal QCQP is exact
by considering at most $n - 1$ systems.
The exactness of the SDP relaxation for other forest-structured QCQPs can be discussed similarly.

To prove the exactness  of the SDP relaxation with our sufficient conditions for a forest-structured QCQP,
the rank property on the aggregated sparsity matrix in \autoref{lem:lowerbound_rank_tridiagonal}
and the perturbation technique in \autoref{lem:perturbation_exactness} are utilized.
We need to estimate a lower bound for the rank of the dual SDP as in Burer and Ye's work \cite{Burer2019}.
The main idea there  
was to estimate  the upper bound for the rank of a SDP solution $(x^*, X^*)$ by the lower bound for the rank for a dual solution $y^*$ of the dual SDP.
The strong duality of SDP and the Sylvester's inequality on the rank were used in \cite{Burer2019}.
In this work, a lower bound for the rank  is estimated using 
the lower bound for the rank $n - 1$ for a forest-structured QCQP with $n$ variables and nonzero off-diagonal elements. 

We also investigate the exactness of the SDP relaxation for non-tridiagonal QCQPs
by applying our results on tridiagonal QCQPs. For instance, most GTRS's are not tridiagonal QCQPs. However,
the exactness of the SDP relaxation for the GTRS was  known by the S-lemma 
\cite{Polik2007} under the Slater's conditions.
We demonstrate that the  GTRS has the exact SDP relaxation without replying on the S-lemma.
More precisely,
we show the exactness of the SDP relaxation for the GTRS by  our result on tridiagonal QCQPs.
To apply the results on tridiagonal QCQPs,
the GTRS should be transformed to the tridiagonal QCQP.
We improve the simultaneous tridiagonalization technique proposed in~\cite{Sidje2011}
so that all tridiagonal QCQPs constructed from the GTRS
always satisfy the conditions for the exact SDP relaxation.
Similarly, the exactness of other classes of QCQPs can be analyzed by the method presented in this paper.

The rest of this paper is organized as follows.
In section 2, we review   related works on the exact SDP relaxation of QCQPs,
and present some background materials  for the subsequent discussion.
Some basic properties on tridiagonal matrices are also summarized
to discuss on tridiagonal QCQPs.
Sections 3 and 4 include the main results of this paper.
In section 3,
the main results for forest-structured QCQPs are described,
and  sufficient conditions for the exactness are also proposed. % in section 3.
A perturbation technique used in the proofs for these conditions is also shown.
In section 4, we first describe the simultaneous tridiagonalization technique. Then,
the conditions presented in section 3 are applied to non-tridiagonal QCQPs.
We also present an alternative proof on the exactness of Generalized Trust-Region Subproblem 
in  section 4. 
Finally, we conclude in section 5.

%\input sect2.tex

%!TEX root = ./main.tex
\section{Preliminaries}
We start by introducing notation and symbols used in this paper.

\subsection{Notation and symbols} % {{{
\begin{itemize}
    \item
        $\bR^n$ and $\bS^n$  denote
        the $n$-dimensional Euclidean space,
        the space of $n \times n$ symmetric matrices, respectively.
        The notation $M \succeq O$ and $M \succ O$ mean that the matrix $M$ is
        positive semidefinite and positive definite, respectively.
    \item
        $0_n \in \bR^n$ denotes the zero vector of length $n$.
        $I_n \in \bR^{n \times n}$ denotes the  $n \times n$ identity matrix.
    \item
        For $M, N \in \bS^n$,
        $\innerp{M}{N}$ denotes the Frobenius inner product of $M$ and $ N$,
        i.e., $\innerp{M}{N} \coloneqq \trace(\trans{M}N)= \sum_{i, j} M_{ij} N_{ij}$.
    \item
        For $M \in \bS^n$,
        $\|M\|_\mathrm{max}$ denotes the maximum norm of $M$, i.e.,
        $\|M\|_\mathrm{max} =  \max_{1 \leq i, j \leq n} |M_{ij}|$.
    \item
        $[n]$ is a shorthand notation for
        $\left\{i \in \bN \,\middle|\, 1 \leq i \leq n\right\}$.
    \item
        $\submatrix{M}{I}{J}$ denotes the submatrix of $M$
        constructed by collecting the rows of $M$ indexed by $I \subset [n]$ and the columns of $M$ indexed by $J \subset [n]$.
        We use $\submatrix{M}{I}{}$ for $\submatrix{M}{I}{I}$.
\end{itemize}

Let $M = [m_{ij}] \in \mathbb{S}^n$ for $1 \leq i,j \leq n$.
We use $[m_{ij}]$ to denote a matrix $M \in \SMAT^n$ whose $(i,j)$-th element is
$m_{ij}$ and also use $[Q^0]_{ij}$ to denote the $(i,j)$-th element of a matrix $Q^0 \in\SMAT^n$.

\begin{define}
    A finite set $\{a_1, \ldots, a_n\} \subset \bR$
    is called sign-definite with respect to $\bR$
    if its members are either all nonnegative or all nonpositive,
    i.e., $a_ia_j \geq 0$ for any $i, j \in [n]$.
\end{define}
\noindent
For example,
the set $\{0, 100, 0, 2\}$ is sign-definite
while sets $\{0, 100, 0, -2\}$ and $\{-1, 1\}$ are not sign-definite.
% }}}

%-========================================================================================
\subsection{Tridiagonal matrices} % {{{
%A tridiagonal matrix is a sparse matrix
%which has nonzero elements on the main diagonal,
%or the place above or below the main diagonal only.
%More precisely,
A matrix $M = [m_{ij}] \in \bS^n$ is called tridiagonal
if all elements $m_{ij}$ are zero for $i, j \in [n]$ satisfying $|i - j| \geq 2$.
%\begin{equation*}
%M \coloneqq \begin{bmatrix}
%m_{11} & m_{12} &  & & \\
%m_{12} & m_{22} & m_{23} & & \\
%& \ddots & \ddots & \ddots & \\
%& & m_{n-1,n-2} & m_{n-1,n-1} & m_{n-1,n} \\
%& & & m_{n-1,n} & m_{nn} \\
%\end{bmatrix}.
%\end{equation*}
% Since $M$ is symmetric, $m_{i+1,1}$ \tgreen{and  $m_{i,i+1}$ are the same.}
%
%where blank elements are $0$.
%
%The diagonal elements $m_{11}, \ldots, m_{nn}$ are called the main diagonal.
%In particular, tridiagonal matrix $M$ has the other two diagonals
%$m_{12}, m_{23}, \ldots, m_{n-1,n}$ and $m_{21}, m_{32}, \ldots, m_{n,n-1}$.
%The elements that lie directly above the main diagonal are called superdiagonal
%whereas the elements  below  the main diagonal are called subdiagonal.
%The off-diagonal elements for tridiagonal matrices
%mean both the superdiagonal and the subdiagonal.
We use $\idiagonal{n}$ and $\ioffdiagonal{n}$ to represent the
index sets for the main diagonal and  off-diagonal elements  % the superdiagonal 
of $n \times n$ matrices, respectively,
i.e.,
\begin{align*}
\idiagonal{n} \coloneqq& \{(i, i) \,|\, i = 1, \ldots, n\}, \\
\ioffdiagonal{n} \coloneqq& \{(i, i + 1) \,|\, i = 1, \ldots, n - 1\}.
\end{align*}
Since $M \in \bS^n$,
    $\ioffdiagonal{n}$ contains only the indices for the upper triangular elements.

We discuss a method to estimate a lower bound on the rank of tridiagonal matrices.
For general matrices, this estimation is generally hard.
In the case of a diagonal matrix, %a simple matrix to compute its rank, 
we know that its rank equals
the number of nonzero elements on its main diagonal.
%By some useful properties of 
For a tridiagonal matrix, % also have relatively useful properties,
we can show that its rank can be bounded from below by the number of off-diagonal nonzero elements.
The following lemma is immediately obtained from the result of \cite{JOHNSON99}.

\begin{lemma} \label{lem:lowerbound_rank_tridiagonal}
    Let $M \in \mathbb{S}^n$ be a tridiagonal matrix.
    If all the superdiagonal elements of $M$ are nonzeros,
    then $\rank{M} \geq n - 1$.
\end{lemma}
\begin{comment}
\begin{proof} % {{{
    Let $\lambda$ be an arbitrary eigenvalue of $M$.
    We consider an $(n - 1) \times (n - 1)$ matrix $L$ constructed by
    removing the first row and the $n$th column from $M - \lambda I$,
    i.e.,
    \begin{equation*}
        L \coloneqq \submatrix{\left(M - \lambda I\right)}{\{2, \ldots, n\}}{\{1, \ldots, n - 1\}} =
        \begin{bmatrix}
            m_{12} & m_{22} - \lambda & m_{23}           &        & \\
                   & m_{23}           & m_{33} - \lambda & \ddots & \\
                   &                  & m_{34}           & \ddots & m_{n - 2, n - 1} \\
                   &                  &                  & \ddots & m_{n - 1, n - 1} - \lambda \\
                   &                  &                  &        & m_{n - 1, n} \\
        \end{bmatrix}.
    \end{equation*}
    %
    Then $L$ is a nonsingular matrix
    as all elements on the main diagonal
    are nonzeros.
    Since $L$ is a submatrix of $M - \lambda I$,
    the rank of $M - \lambda I$ must be greater than or equal to $n - 1$,
    and hence both the geometric and the algebraic multiplicity of $\lambda$ are $1$.
    It implies that
    $\rank(M) = n - 1$ if the eigenvalues of $M$ contains $0$;
    otherwise $\rank(M) = n$.
\end{proof} % }}}
\end{comment}

For symmetric positive semidefinite matrices $\bS^n \ni M \succeq O$, it is difficult to determine whether the positive semidefiniteness
is maintained after replacing some of off-diagonal elements with zeros.
However, in the case of positive semidefinite tridiagonal matrices, 
we show in the following lemma
that  they remain to be positive semidefinite even if some of  off-diagonal elements are replaced by zeros.

\begin{lemma} \label{lem:tridiagonal_keeps_psd}
    Let $M = [m_{ij}] \in \bS^n$ be a positive semidefinite and tridiagonal  
    matrix.
    For a subset $E \subset \ioffdiagonal{n}$,
    let $L = [\ell_{ij}] \in \bS^n$ be the tridiagonal matrix
    constructed by replacing the matrix elements of $M$ indexed by $E$ with zeros,
    i.e.,
    \begin{equation*}
        \ell_{ij} := \begin{cases}
            0      & \text{if $(i, j) \in E$ or $(j, i) \in E$}, \\
            m_{ij} & \text{otherwise.}
        \end{cases}
    \end{equation*}
    Then, $L \succeq O$.
\end{lemma}
\begin{proof} % {{{
    We use  induction on the size of the set $E$.
    Let us first consider the case $|E| = 1$.
    Then, $E$ has one element $(i, i + 1) \in E$ for some $i$.
    For any $I \subset [n]$, we have the following  two cases: 
    \begin{enumerate}[label={(\alph*)}]
        \item
            If $i \not\in I$ or $i + 1 \not\in I$,
            the principal $\submatrix{L}{I}{}$ does not have $(i, i+1)$-th and $(i+1, i)$-th elements of $M$,
            therefore
            \begin{equation*}
                \mathrm{det}\left(\submatrix{L}{I}{}\right) = \mathrm{det}(\submatrix{M}{I}{}) \geq 0.
            \end{equation*}
        \item
            If $i \in I$ and $i + 1 \in I$,
            the principal $\submatrix{L}{I}{}$ has $(i, i+1)$-th and $(i+1, i)$-th elements of $M$
            replaced by zeros.
            Since the submatrix $\submatrix{L}{I}{}$ is a block diagonal matrix with two blocks,
            \begin{align*}
                \mathrm{det}\left(\submatrix{L}{I}{}\right)
        &= \mathrm{det}\left(L_{I \cap \{1, \ldots, i\}}\right)
        \mathrm{det}\left(L_{I \cap \{i+1, \ldots, n\}}\right) \\
        &= \mathrm{det}\left(M_{I \cap \{1, \ldots, i\}}\right)
        \mathrm{det}\left(M_{I \cap \{i+1, \ldots, n\}}\right) \\
        &\geq 0,
            \end{align*}
            where the last inequality follows from the fact that
            all the principal minors of $M$ are nonnegative.
    \end{enumerate}
    Since all the principal minors of $L$ are nonnegative,
    $L \succeq O$ follows.

    Suppose the result is true for $|E| = k - 1$,
    and consider the case  $|E| = k$.
    The set $E$ can be divided into two sets:
    $F \coloneqq \{(j, j + 1)\}$ and $E \setminus F$.
    Let $N \in \bS^n$ be a tridiagonal matrix
    constructed by replacing the $(j,j+1)$ and $(j+1,j)$ elements of $M$ with zeros.
    Then, from the  case mentioned above, $N \succeq O$ holds.
    Since $\left|E \setminus F\right| = k - 1$,
    by the induction hypothesis,
    we have $L \succeq O$.
\end{proof} % }}}

In our proof of \autoref{lem:tridiagonal_keeps_psd},
 the nonnegativeness of all the principal minors of $L$
was shown based on the fact that
they  are given by the product of at most two principal minors of $M$.
The nonnegativeness can be also shown
by  a representation of the determinant of the tridiagonal matrix.
More details can be found in Corollary 2.2 of \cite{El-Mikkawy2003}.
%-======================================================================================== }}}

%================================================================================
\subsection{Aggregate sparsity matrix and forest-structured QCQPs} \label{ssec:aggregate_sparsity} % {{{

In \autoref{lem:lowerbound_rank_tridiagonal},
the rank of tridiagonal matrices is discussed without taking the
positive semidefiniteness into account.
As the rank of positive semidefinite matrices plays a crucial role
to extend the result in \autoref{lem:lowerbound_rank_tridiagonal}
to more general matrices,
we briefly introduce the aggregate sparsity matrix
and discuss the relation between the aggregate sparsity matrix and forest-structured QCQPs.

To construct the aggregate sparsity matrix from
the SDP relaxation \eqref{eq:sdr}, we define 
an aggregate sparsity graph $G(\VC,\EC)$ as a graph
with the set of vertices $\VC = [n]$ 
and  the set of edges 
\begin{equation} \label{eq:aggregate_sparsity_pattern}
    \EC = \left\{(i,j)  \in \VC \times \VC \ \middle| \ 
    \ [Q^p]_{ij} \neq 0 \ \text{for some} \ p \in \{0\} \cup [m]\right\}.
\end{equation}
The sparsity encoded in $\EC$ is called the aggregate sparsity pattern.
The aggregate sparsity matrix $R \in \SMAT^n$ corresponding to $G(\VC,\EC)$ is defined as 
\begin{equation*}
    R_{ij} = \left\{ \begin{array}{cl} *  & \mbox{if } (i,j) \in \EC, \\
    0 & \mbox{otherwise}, \end{array} \right.
\end{equation*}
where $*$ is an arbitrary nonzero real number.

A  graph is called a forest if it has no cycles in the upper triangular part $\EC  \cap \left\{(i,j) \in \VC \times \VC \ \middle| \ i < j \right\}$.
A connected forest is called a tree.
We also call QCQPs \eqref{eq:qcqp} as forest-structured QCQPs
if their aggregated sparsity graphs $G(\VC,\EC)$ are forests.
Thus,
the graph $G(\VC,\EC)$ obtained from forest-structured QCQPs
consists of one or more trees.
Obviously, the tridiagonal QCQP is a subclass of  forest-structured QCQP.

For a given symmetric matrix $M = [m_{ij}] \in \bS^n$,
a sparsity structure graph $G(\VC,\EC)$ can also be defined 
as an  graph with the set of vertices $\VC = [n]$ and the set of edges
\[
    \EC = \left\{(i,j)  \in \VC \times \VC \ \middle| \ m_{ij} \neq 0\right\}.
\]
We call symmetric matrices $M \in \bS^n$ as forest-structured matrices
if their sparsity structure graphs have no cycles.
%When a QCQP corresponds to 
For $G(\VC,\EC)$ corresponding to a QCQP,
the sparsity structure graph of  $Q^p$ is a subgraph of $G(\VC,\EC)$ for any $p \in \{0\} \cup [m]$.
As a result,
all the matrices $Q^p\ (p \in \{0\} \cup [m])$ in forest-structured QCQPs must be forest-structured matrices.

We consider the following index sets in the subsequent discussion for the forest-structured matrices and QCQPs:
\begin{align*}
    \idiagonal{n} \coloneqq& \{(i, i ) \,|\, i = 1, \ldots, n \} \\
    \ioffdiagonal{n} \coloneqq& \{(i, j) \in \EC \,|\, i < j\},
\end{align*}
where $(i,j)$ is the index for the nonzero element of the aggregate sparsity matrix.
By an appropriate permutation on $[n]$, we can assume that 
$n \in \VC$ is the root of a tree graph.
If $j$  corresponds to the parent node of $i$ in a tree graph,
then it is uniquely determined for each $i \in \VC \backslash \{n\}$.

As an immediate consequence of Corollary 3.9 of \cite{JOHNSON03},
we can establish  that
any symmetric positive semidefinite matrix whose graph is a tree has rank at least $n-1$, as described in the following lemma.
We will use this fact to estimate the rank of SDP solutions
of  SDP relaxations in \autoref{sec:exactness_of_triqcqp}.
\begin{lemma} \label{lem:lowerbound_rank_treestructured} \cite[Corollary 3.9]{JOHNSON03}
    Let $M \in \bS^n$ be a positive semidefinite and forest-structured matrix.
    If the sparsity structure graph of $M$ is connected 
    and all the off-diagonal elements in $\ioffdiagonal{n}$ for $M$ are nonzeros,
    then $\rank{M} \geq n - 1$.
\end{lemma}

As an example of forest-structured QCQPs, we consider 
an arrow-type matrix of form: 
\begin{equation*}  \label{ARROW}
    V = \left(\begin{array}{ccccc}
            v_1 &  &  &  & w_1 \\
                & v_2 & & & w_2 \\
                & & \ddots & & \vdots \\
                & & & v_{n-1} & w_{n-1} \\
            w_1 & w_2 & \cdots & w_{n-1} & v_n 
    \end{array}\right).
\end{equation*}
For the arrow-type matrix,
the index sets $\idiagonal{n}$ and $\ioffdiagonal{n}$ are given as
\begin{align*}
    \idiagonal{n} \coloneqq& \{(i, i) \,|\, i = 1, \ldots, n\} \\
    \ioffdiagonal{n} \coloneqq& \{(i, n ) \,|\, i = 1, \ldots, n - 1\}.
\end{align*}
By \autoref{lem:lowerbound_rank_treestructured},
if $V \succeq O$  
and $w_1,\ldots,w_{n-1}$ are nonzeros, then the rank of $V$ is at least $n-1$.

\autoref{lem:tridiagonal_keeps_psd} can also be extended to forest-structured matrices.
\begin{lemma} \label{lem:treestructured_keeps_psd}
    Let $M = [m_{ij}] \in \bS^n$ be a positive semidefinite and forest-structured 
    matrix.
    For a subset $E \subset \ioffdiagonal{n}$,
    let $L = [\ell_{ij}] \in \bS^n$ be the forest-structured matrix
    constructed by replacing the matrix elements of $M$ indexed by $E$ with zero,
    i.e.,
    \begin{equation*}
        \ell_{ij} := \begin{cases}
            0      & \text{if $(i, j) \in E$ or $(j, i) \in E$}, \\
            m_{ij} & \text{otherwise.}
        \end{cases}
    \end{equation*}
    Then, $L \succeq O$.
\end{lemma}

\begin{proof} 
    It suffices to consider the case $|E| = 1$
    since  similar arguments to the proof of \autoref{lem:tridiagonal_keeps_psd} can be applied.
    In this case, there exists only one element $(i,j) $ in  $E$.  
        By removing the edge $(i, j)$ from $G(\VC,\EC)$ of the forest-structured matrix $M$,
    %implies that
    a tree in $G(\VC,\EC)$ is divided into two trees:
    one with the node $i$, and the other with the node $j$.
    %Now, we also separate 
    The set $\VC$ is also separated  into two sets: 
    $W_1 \subset \VC$, the set of nodes in the 
        component including the node $i$
    in the  graph $G(\VC,\EC \setminus \{(i,j), (j,i)\})$, and
    $W_2 \subset \VC$, the set of  other nodes.
    Without loss of generality,
    we may assume that the indices in $W_1$ and $W_2$ are consecutive integers,
    i.e., there exists a positive number $\ell$ such that $W_1 = [\ell]$ and $W_2 = [n] \setminus [\ell]$.
    For any $I \subset [n]$, we have the following two cases:
    \begin{enumerate}[label={(\alph*)}]
        \item
            If $i \not\in I$ or $j \not\in I$,
            the principal $\submatrix{L}{I}{}$ does not include $(i, j)$-th and $(j, i)$-th elements of $M$,
            then
            \begin{equation*}
                \mathrm{det}\left(\submatrix{L}{I}{}\right) = \mathrm{det}(\submatrix{M}{I}{}) \geq 0.
            \end{equation*}
        \item
            If $i \in I$ and $j \in I$,
            the principal $\submatrix{L}{I}{}$ includes $(i, j)$-th and $(j, i)$-th elements of $M$,
            and their values are zeros.
            Since the submatrix $\submatrix{L}{I}{}$ is a block diagonal matrix with two blocks,
            we have that
            \begin{align*}
                \mathrm{det}\left(\submatrix{L}{I}{}\right)
                &= \mathrm{det}\left(L_{I \cap W_1}\right)
                \mathrm{det}\left(L_{I \cap W_2}\right) \\
                &= \mathrm{det}\left(M_{I \cap W_1}\right)
                \mathrm{det}\left(M_{I \cap W_2}\right) \\
                &\geq 0,
            \end{align*}
            where the last inequality follows from the fact that
            all the principal minors of $M$ are nonnegative.
    \end{enumerate}
    Since all the principal minors of $L$ are nonnegative,
    we have $L \succeq O$.
\end{proof} % }}}

%\input sect3.tex
%!TEX root = ./main.tex
\section{Exactness conditions for forest-structured QCQPs}
\label{sec:exactness_of_triqcqp}

To state our main results on the exact SDP relaxation of forest-structured QCQPs,
some assumptions are necessary. Also, a perturbed QCQP is introduced by
slightly varying the elements of the data matrix of the objective function.
We then present preliminary results on the perturbed QCQP.
By analyzing the conditions under which the SDP relaxation has a rank-1 solution,
we discuss the exact SDP relaxation for forest-structured QCQPs.

We assume the following  for a given QCQP and its SDP relaxation:
\begin{assumption} \label{asm:assumptions} \quad\par
    \begin{enumerate}[label={(\alph*)}]
        \item \label{asm:assumption_1}
            There exists a feasible point for \eqref{eq:qcqp}.
        \item \label{asm:assumption_2}
            There exists $\bar{y} \geq 0$ satisfying $\sum_{p = 1}^m \bar{y}_pQ^p \succ O$.
        \item \label{asm:assumption_3}
            There exists an interior feasible points for \eqref{eq:sdr}.
    \end{enumerate}
\end{assumption}
\noindent
We note that these assumptions were also used for the diagonal QCQPs to establish the results on the exact SDP relaxations in~\cite{Burer2019}.
  \autoref{asm:assumptions} \ref{asm:assumption_2} 
can be also represented as $\bar{y} \leq 0$ and $\sum_{p = 1}^m \bar{y}_pQ^p \prec O$. % in some papers.
By \autoref{asm:assumptions} \ref{asm:assumption_1} and \ref{asm:assumption_2}, we see that
the feasible sets of \eqref{eq:qcqp} and \eqref{eq:sdr} are bounded,
and a solution to \eqref{eq:qcqp} exists.
In fact,  multiplying $\trans{x}Q^px + 2\trans{q_p}x \leq b_p$ by $\bar{y}_p$,
and adding these for $p$, we have
\begin{equation*}
    \trans{x}\left(\sum_{p = 1}^m \bar{y}_pQ^p\right)x
    + 2\trans{\left(\sum_{p = 1}^m \bar{y}_pq_p\right)}x \leq \trans{b}\bar{y}.
\end{equation*}
Thus, all the feasible points of \eqref{eq:qcqp}  are in the ellipsoid given by the above inequality.
It also implies that
all the feasible points of \eqref{eq:sdr} is bounded as:
\begin{equation*}
    \innerp{\left(\sum_{p = 1}^m \bar{y}_pQ^p\right)}{X}
    + 2\trans{\left(\sum_{p = 1}^m \bar{y}_pq_p\right)}x \leq \trans{b}\bar{y}.
\end{equation*}
By \autoref{asm:assumptions} \ref{asm:assumption_2} and \ref{asm:assumption_3},
the strong duality holds for the primal SDP \eqref{eq:sdr}.

The homogeneous form of \eqref{eq:qcqp} can be expressed as the following QCQP
with $n + 1$ variables and $m + 2$ inequality constraints:
\begin{mini*}
    {}{\trans{z}\bar{Q}^0z}{}{}
    \addConstraint{\trans{z}\bar{Q}^pz}{\leq b_p, \quad}{p = 1, \ldots, m}
    \addConstraint{\trans{z}E_{11}z}{\leq 1, \quad}{-\trans{z}E_{11}z \leq -1,}
\end{mini*}
with variable $z \in \bR^{n + 1}$ where
\begin{equation*}
    \bar{Q}^p \coloneqq
    \begin{bmatrix} 0 & \trans{q_p} \\ q_p & Q^p \end{bmatrix},
\end{equation*}
and $E_{ij}$ is an $n \times n$ matrix given by
\begin{equation*}
    [E_{ij}]_{qr} = \begin{cases}
        1 & \text{if $i = q$ and $j = r$,} \\
        0 & \text{otherwise.}
    \end{cases}
\end{equation*}
We note that the homogeneous form does not include the linear terms $\trans{q_p}x$ for all $p$,
and the objective function and the constraints are in quadratic forms in the variables.
%that is, the linear terms $\trans{q_p}x$ for all $p$ are vanished.
By the last two inequalities,  $z_1 = 1$ or $-1$.
Although the homogeneous QCQP has a simpler form than the original \eqref{eq:qcqp},
any solution $z^*$ of the homogeneous QCQP recovers a solution
$x^* = \trans{[z^*_2 / z^*_1, \ldots, z^*_{n + 1} / z^*_1]}$ of \eqref{eq:qcqp}.
As a result,  we may assume,
without loss of generality,
the following condition:
\begin{assumption} \label{asm:assumptions2} \quad\par
    \begin{enumerate}[label={(\alph*)},start=4]
        \item \label{asm:assumption_4}
            $q_0, q_1, \ldots, q_m$ are zero vectors.
    \end{enumerate}
\end{assumption}

We now consider a homogeneous QCQP: % in this section:
\begin{mini}
    {}{\trans{x}Q^0x}{\label{eq:hqcqp}}{}
    \addConstraint{\trans{x}Q^px}{\leq b_p, \quad}{p = 1, \ldots, m.}
\end{mini}
\noindent
Then,  the SDP relaxation of  QCQP \eqref{eq:hqcqp} is described as
\begin{mini}
    {}{\innerp{Q^0}{X}}{\label{eq:hsdr}}{}
    \addConstraint{\innerp{Q^p}{X}}{\leq b_p, \quad}{p = 1, \ldots, m}
    \addConstraint{X}{\succeq O,}{}
\end{mini}
and its dual SDP problem is:
\begin{maxi}
    {}{-\trans{b}y}{\label{eq:hsdrd}}{}
    \addConstraint{S(y) := Q^0 + \sum_{p = 1}^m y_p Q^p}{\succeq O}
    \addConstraint{y}{\geq 0.}
\end{maxi}

%=========================================================================
\subsection{Perturbed QCQPs} % {{{
To perturb the original QCQP \eqref{eq:hqcqp},
we let $P \in \bS^n$ be a nonzero matrix,
and let $\varepsilon > 0$ represent how much perturbation is added to the objective function.
When a perturbation $\varepsilon P$ is added to the objective function of \eqref{eq:hqcqp},
we have  an $\varepsilon$-perturbed QCQP:
\begin{mini}
    {}{\trans{x}\left[Q^0 + \varepsilon P\right]x}{\label{eq:perturbed_hqcqp}}{}
    \addConstraint{\trans{x}Q^px}{\leq b_p, \quad}{p = 1, \ldots, m.}
\end{mini}
The  SDP relaxation of \eqref{eq:perturbed_hqcqp}  can be written  as:
\begin{mini}
    {}{\innerp{\left[Q^0 + \varepsilon P\right]}{X}}{\label{eq:perturbed_sdr}}{}
    \addConstraint{\innerp{Q^p}{X}}{\leq b_p, \quad}{p = 1, \ldots, m}
    \addConstraint{X}{\succeq O.}{}
\end{mini}
 QCQP \eqref{eq:hqcqp} and
its $\varepsilon$-perturbed QCQP \eqref{eq:perturbed_hqcqp}
have the same feasible set since %the perturbation affects
only the objective function %of \eqref{eq:perturbed_hqcqp} 
is perturbed.
Similarly, their SDP relaxations have the same feasible set.
As  the feasible sets of QCQP \eqref{eq:hqcqp} and its SDP relaxation
are bounded by \autoref{asm:assumptions},
 the feasible sets for the perturbed problems \eqref{eq:perturbed_hqcqp} and \eqref{eq:perturbed_sdr} are also bounded.
The $\varepsilon$-perturbed problem  \eqref{eq:perturbed_hqcqp} will be used % useful
to check whether the SDP relaxation of \eqref{eq:hqcqp} is exact.

In the following lemma,  we show that the exactness of the SDP relaxation for the original problem \eqref{eq:hqcqp}  can be determined by
that of perturbed problems \eqref{eq:perturbed_hqcqp} and \eqref{eq:perturbed_sdr}.
%the exactness of perturbation problems can decide
%that of the original problem.
%
\begin{lemma} \label{lem:perturbation_exactness}
    Let $P \neq O$ be an $n \times n$ nonzero matrix, and
    $\{\varepsilon_t\}_{t = 1}^\infty$ be a %n arbitrary 
    sequence such that  $\lim_{t \to \infty} \varepsilon_t = 0$.
    If the SDP relaxation of the $\varepsilon_t$-perturbed problem \eqref{eq:perturbed_hqcqp}
    is exact for all $t = 1, 2, \ldots$,
    then the original problem \eqref{eq:hqcqp} also has an exact SDP relaxation.
\end{lemma}
\begin{proof} % {{{
    Let $\Gamma$ and $\Delta$ be
    the feasible sets of \eqref{eq:hqcqp} and \eqref{eq:hsdr}, respectively:
    \begin{align*}
        \Gamma \coloneqq & \left\{
            x \in \bR^n \,\middle|\,
            \innerp{Q^p}{(x\trans{x})} \leq b_p, \quad p = 1, \ldots, m\right\}, \\
        \Delta \coloneqq & \left\{
            X \in \bS^n \,\middle|\,
        X \succeq O,\; \innerp{Q^p}{X} \leq b_p, \quad p = 1, \ldots, m\right\}.
    \end{align*}
    Note that $\Gamma$ is a closed set and
    both $\Gamma$ and $\Delta$ are bounded by \autoref{asm:assumptions}.
    Thus, $\Gamma$ is a compact set in $\bR^n$.
    For any $t \geq 1$,
    let $x_t$ and $X_t$ be optimal solutions
    of \eqref{eq:perturbed_hqcqp} and \eqref{eq:perturbed_sdr}
    satisfying $x_t\trans{x_t} = X_t$,
    which follows from the assumption on the exactness of the relaxation of \eqref{eq:perturbed_hqcqp}.
    As a result, a sequence as $\{x_t\}_{t = 1}^\infty$ can be defined.

    Since  %there is no difference between 
    the feasible sets of \eqref{eq:hqcqp} and \eqref{eq:perturbed_hqcqp} are identical,
    we have $x_t \in \Gamma$.
    From the compactness of $\Gamma$,
    there exists $x_{\lim} \coloneqq \lim_{t \to \infty} x_t$ in $\Gamma$.
    As $X_{\lim} \coloneqq x_{\lim}\trans{x_{\lim}} \in \Delta$
    by the relationship between $\Gamma$ and $\Delta$,
    the rank-1 matrix $X_{\lim}$ is also feasible for \eqref{eq:hsdr}.

    To show that $X_{\lim}$ is an optimal solution of \eqref{eq:hsdr},
    we assume that there exists another feasible $X_\mathrm{opt} \neq X_{\lim}$
    such that $\nu \coloneqq \innerp{Q^0}{X_{\lim}} - \innerp{Q^0}{X_\mathrm{opt}} > 0$.
    Since $\Delta$ is bounded,
    there exists $\mu$ such that $\|X\|_{\max} < \mu$ for any $X \in \Delta$,
    which implies $\|x\trans{x}\|_{\max} < \mu$ for any $x \in \Gamma$.
    For a sufficiently large $t$ satisfying
    \begin{align*}
      &  \varepsilon_t < \frac{\nu}{4n^2\mu\|P\|_{\max}} \quad \text{ and } \quad
        \|x_t\trans{x_t} - X_{\lim}\|_{\max} < \frac{\nu}{2n^2\|Q^0\|_{\max}},
    \end{align*}
    we have
    \begin{align*}
    & \innerp{Q^0}{\left(x_t\trans{x_t} - X_{\lim}\right)} > -\frac{\nu}{2}, \\
    & \varepsilon_t\innerp{P}{(x_t\trans{x_t})} > -\frac{\nu}{4}, \\
    & \innerp{Q^0}{X_{\lim}} = \innerp{Q^0}{X_\mathrm{opt}} + \nu, \\
    & \frac{\nu}{4} > \varepsilon_t\innerp{P}{X_\mathrm{opt}}.
    \end{align*}
    %
    %Summing up them gives
    Consequently, adding these inequalities and the equality,
    \begin{equation*}
        \innerp{(Q^0 + \varepsilon_t P)}{(x_t\trans{x_t})} >
        \innerp{(Q^0 + \varepsilon_t P)}{X_\mathrm{opt}},
    \end{equation*}
    which contradicts the optimality of $x_t\trans{x_t}$ in \eqref{eq:perturbed_sdr}. The desired result follows.
\end{proof} % }}}
\noindent
Zhou et al.~\cite[Lemma 1]{Zhou2019} focused on a specific QCQP arising from the oprimal power flow problem,
and proved that the exactness of its SDP relaxation can be determined
by the SDP relaxation of its $\varepsilon$-perturbed problems.
\autoref{lem:perturbation_exactness} is valid for a slightly more general case,
which requires only Assumptions~\ref{asm:assumptions} and \ref{asm:assumptions2}.
% }}}

%==================================================================================

\subsection{Main results} 
We present our main results on the exactness of SDP relaxations %with some conditions 
for forest-structured QCQPs.

For any fixed indices $k, \ell \in [n]$,
we define the system \eqref{eq:system}:
\begin{equation} \tag{$\FC_{{k}{\ell}}$} \label{eq:system}
    y \geq 0, \quad
    S(y) \succeq O, \quad [S(y)]_{k\ell} = 0,
\end{equation}
where $[S(y)]_{k\ell}$ represents the $(k, \ell)$-th element of $S(y)$
 defined in \eqref{eq:hsdrd}.
Since $k,\ell\in[n]$, we can construct $n^2$ systems \eqref{eq:system}.

If  for some $\bar{k}$ and $\bar{\ell}$ $(1\leq \bar{k},\bar{\ell} \leq n)$ and 
for a feasible point $y$ of  \eqref{eq:hsdrd} the system $(\FC_{ {\bar{k}}{\bar{\ell}} })$ is infeasible,
 then the value $[S(y)]_{\bar{k}\bar{\ell}}$ must be nonzero.

\begin{theorem} \label{thm:main}
    Suppose \eqref{eq:hqcqp} is a forest-structured QCQP.
    If $(\FC_{{k}{\ell}})$ have no feasible solutions for all $(k, \ell) \in \ioffdiagonal{n}$,
    then the SDP relaxation \eqref{eq:hsdr} is exact.
\end{theorem}
To prove the main result described in \autoref{thm:main}, 
we impose a condition in Lemma \ref{lem:main_wo_sparsity}
for the exact SDP relaxation:
the aggregate sparsity graph should be connected
and $(\FC_{{k}{\ell}})$ should have no solutions for any $(k, \ell) \in \ioffdiagonal{n}$.
We note that the system $(\FC_{{k}{\ell}})$ %for $(k, \ell) \in \ioffdiagonal{n}$
must be tested for feasibility for all $(k, \ell) \in \ioffdiagonal{n}$ in \autoref{thm:main}.
We need additionally to examine whether
the aggregate sparsity graph is connected in \autoref{lem:main_wo_sparsity}.
Thus, the sufficient condition presented in \autoref{thm:main} can be applied to 
more general QCQPs than the one in \autoref{lem:main_wo_sparsity}.
Lemma \ref{lem:main_wo_sparsity} is followed by
   a proof for \autoref{thm:main} by relaxing the sufficient condition in \autoref{lem:main_wo_sparsity}.
\begin{lemma} \label{lem:main_wo_sparsity}
    Suppose \eqref{eq:hqcqp} is a forest-structured QCQP
    and the aggregate sparsity graph $G(\VC,\EC)$ of its SDP relaxation
    defined in \eqref{eq:aggregate_sparsity_pattern} is connected.
    If $(\FC_{{k}{\ell}})$ have no feasible solutions
    for all $(k, \ell) \in \ioffdiagonal{n}$,
    then the SDP relaxation \eqref{eq:hsdr} is exact.
\end{lemma}
\begin{proof} % {{{
    Let $X^*$ be any optimal solution for \eqref{eq:hsdr}.
    By \autoref{asm:assumptions},
    there exists an optimal solution $y^*$ for \eqref{eq:hsdrd}.
    Since $y^* \geq 0$ and $S(y^*) \succeq O$,
    we have $S(y^*)_{k\ell} \neq 0$ for every $(k, \ell) \in \ioffdiagonal{n}$ by the assumption.
    This implies that all the off-diagonal elements on $\ioffdiagonal{n}$ of the forest-structured matrix $S(y^*)$ are nonzeros,
    thus $\rank{S(y^*)} \geq n - 1$ by \autoref{lem:lowerbound_rank_treestructured}.
    Since \autoref{asm:assumptions} \ref{asm:assumption_2} and \ref{asm:assumption_3}  hold,
    $X^*S(y^*) = O$ by the strong duality.
    From the Sylvester's rank inequality~\cite{Anton2014},
        $\rank{X^*} + \rank{S(y^*)} \leq n + \rank{X^*S(y^*)}$ holds for $X^*$ and $S(y^*)$.
    Therefore, $\rank{X^*} \leq n + \rank{O} - \rank{S(y^*)} \leq n + 0 - (n - 1) = 1$.
\end{proof} % }}}

\begin{proof} (\autoref{thm:main}) \\ %\par % {{{
    Let $G(\VC,\EC)$ be the aggregate sparsity graph of the SDP relaxation
    defined in \eqref{eq:aggregate_sparsity_pattern} for the forest-structured QCQP,
    and let $\kappa$ denote the number of connected components of $G(\VC,\EC)$.
    Since $G(\VC,\EC)$ consists of one or more trees,
    we can construct a set $\DC$ with $\kappa - 1$ edges
    such that $G(\VC,\EC \cup \DC)$ is a tree  (i.e., a connected graph with no cycles).
    Let $P \coloneqq \sum_{(i, j) \in \DC} (E_{ij} + E_{ji})$ be
    a perturbation matrix with the $n \times n$ matrices $E_{ij}$'s.
      With $\varepsilon > 0$, consider the  $\varepsilon$-perturbed QCQP \eqref{eq:perturbed_hqcqp} with this $P$.
    Obviously, the aggregate sparsity graph of the
    $\varepsilon$-perturbed QCQP \eqref{eq:perturbed_hqcqp} is $G(\VC,\EC \cup \DC)$.
    The system \eqref{eq:system} that corresponds to
    the $\varepsilon$-perturbed QCQP \eqref{eq:perturbed_hqcqp} can be written as:
    \begin{subequations} \label{eq:system_offd_perturbed_hqcqp_all}
        \begin{gather}
            y \geq 0, \quad
            Q^0 + \sum_{(i, j) \in \DC} \varepsilon (E_{ij} + E_{ji})
            + \sum_{p = 1}^m y_p Q^p \succeq O, \label{eq:system_offd_perturbed_hqcqp_a} \\
            \left[Q^0 + \sum_{(i, j) \in \DC} \varepsilon (E_{ij} + E_{ji})\right]_{k\ell}
            + \sum_{p = 1}^m y_p [Q^p]_{k\ell} = 0 \label{eq:system_offd_perturbed_hqcqp_b}.
        \end{gather}
    \end{subequations}
    For the exact SDP relaxation  of \eqref{eq:perturbed_hqcqp},
    we need to show that
    \eqref{eq:system_offd_perturbed_hqcqp_all} has no feasible solutions for all $(k, \ell) \in \ioffdiagonal{n} \cup \DC$.
    First, suppose $(k, \ell) \in \DC$.
    Since $[Q^p]_{k\ell} = 0 \; (\forall p = 0, 1, \ldots, m)$,
    the left hand side of \eqref{eq:system_offd_perturbed_hqcqp_b} becomes %is expressed as
    \begin{align*}
       % \{\text{LHS of \eqref{eq:system_offd_perturbed_hqcqp_b}\}}
      %  & [Q^0]_{k\ell} + \sum_{(i, j) \in \DC} \varepsilon [E_{ij} + E_{ji}]_{k\ell} + \sum_{p = 1}^m y_p [Q^p]_{k\ell} \\
        & \varepsilon \sum_{(i, j) \in \DC} [E_{ij} + E_{ji}]_{k\ell} =\varepsilon >0.
      %  &= \varepsilon > 0.
    \end{align*}
    We have shown that
    \eqref{eq:system_offd_perturbed_hqcqp_b} does not hold for any $y \geq 0$.

    Next, we suppose $(k, \ell) \in \ioffdiagonal{n}$.
    Assume that \eqref{eq:system_offd_perturbed_hqcqp_all} has a solution $\hat{y}$.
    Then,
    \begin{equation*}
        \hat{y} \geq 0, \quad
         Q^0 + \sum_{(i, j) \in \DC} \varepsilon (E_{ij} + E_{ji}) + \sum_{p = 1}^m \hat{y}_p Q^p \succeq O, \quad
        [Q^0]_{k\ell} + \sum_{p = 1}^m \hat{y}_p [Q^p]_{k\ell} = 0.
    \end{equation*}
    We now define a forest-structured matrix $\overline{S} \in \SMAT^n$ as: 
    \begin{equation*}
        [\overline{S}]_{qr} = \begin{cases}
            \left[Q^0 + \sum_{(i, j) \in \DC} \varepsilon (E_{ij} + E_{ji}) + \sum_{p = 1}^m \hat{y}_p Q^p\right]_{qr} & (q, r) \not\in \DC, \\
            0 & (q, r) \in \DC.
        \end{cases}
    \end{equation*}
    If $(q, r) \not\in \DC$, by definition, we have $\left[E_{ij} + E_{ji}\right]_{qr} = 0$
    for any $(i, j) \in \DC$.
    If $(q, r) \in \DC$, we have $\left[Q^0 + \sum_{p = 1}^m \hat{y}_p Q^p\right]_{qr} = 0$
    since $[Q^p]_{qr} = 0$ for any $p \in \{0\} \cup [m]$.
    Thus, we obtain that $\overline{S} = Q^0 + \sum_{p = 1}^m \hat{y}_p Q^p$.
    By \autoref{lem:treestructured_keeps_psd},
    it follows that $\overline{S} \succeq O$,
    which implies that $\hat{y}$ solves the following system:
    \begin{equation} \label{eq:system_offd_perturbed_satisfied_removed}
        y \geq 0, \quad
        Q^0 + \sum_{p = 1}^m y_p Q^p \succeq O, \quad
        [Q^0]_{k\ell} + \sum_{p = 1}^m y_p [Q^p]_{k\ell} = 0.
    \end{equation}
    We know that the system \eqref{eq:system_offd_perturbed_satisfied_removed},
    which is equivalent to \eqref{eq:system} for \eqref{eq:hqcqp},
    has no feasible points by the assumption. This is a contradiction.
    Thus, \eqref{eq:system_offd_perturbed_hqcqp_all} has no feasible points.
    By \autoref{lem:main_wo_sparsity},
    the SDP relaxation for \eqref{eq:perturbed_hqcqp} is exact
    for all $\varepsilon > 0$.

    We now take a sequence $\{\varepsilon_t\}_{t = 1}^\infty$ which converges to zero
    so that the SDP relaxation for
    $\varepsilon_t$-perturbed QCQP \eqref{eq:perturbed_hqcqp} is exact for all $t = 1, 2, \ldots$.
    By \autoref{lem:perturbation_exactness},
    we conclude that
    the SDP relaxation for \eqref{eq:hqcqp} is exact.
\end{proof} % }}}
% }}}

%===================================================================================
Theorem~\ref{thm:main}  can be applied to 
the particular case of tridiagonal QCQPs that will be discussed in section \ref{sec:application_nontri}.

Since $\left|\ioffdiagonal{n}\right| \leq n - 1$ by definition,
we must solve at most $n - 1$ systems
in order to determine whether the SDP relaxation of a QCQP is exact.
It may be very time-consuming to solve all $n-1$ systems \eqref{eq:system} due to the positive semidefinite constraint.
To mitigate this difficulty,
 conditions that do not  dependent on \eqref{eq:system} will be discussed in  Corollaries \ref{prop:exact_condition_for_signdefinite} and
  \ref{prop:exact_condition_for_one_equality}.

It was shown in \cite{BOSE2015,kim2003exact,Lavaei2012}
 that QCQP \eqref{eq:qcqp} can be solved exactly by the SDP relaxation if the data matrices of the QCQP is sign-definite.
Using the feasibility of the system \eqref{eq:system}, we provide an alternative proof for 
the exact SDP relaxation of the QCQP \eqref{eq:hqcqp} %admits the exact SDP relaxation
if the set of $(i, j)$-th elements of all matrices is sign-definite
for every superdiagonal index $(i, j)$.
\begin{corollary} \label{prop:exact_condition_for_signdefinite}
    Suppose \eqref{eq:hqcqp} is a forest-structured QCQP.
    If the set $\{[Q^0]_{k\ell}, [Q^1]_{k\ell}, \ldots, [Q^m]_{k\ell}\}$ is sign-definite
    for all $(k, \ell) \in \ioffdiagonal{n}$,
    then the SDP relaxation \eqref{eq:hsdr} is exact.
\end{corollary}
\begin{proof} \par % {{{
    Define $P \coloneqq [p_{ij}] \in \bS^n$ where
    \begin{equation*}
        p_{ij} := \begin{cases}
            +1 & \text{if $[Q^0]_{ij} = 0$ and  $\sum_{p = 1}^m [Q^p]_{ij} \geq 0$}, \\
            -1 & \text{if $[Q^0]_{ij} = 0$ and  $\sum_{p = 1}^m [Q^p]_{ij} < 0$}, \\
            0  & \text{otherwise}.
        \end{cases}
    \end{equation*}
    Let $\{\varepsilon_t\}_{t = 0}^\infty \subset \bR$ be a sequence converging to $0$.
    Consider the $\varepsilon_t$-perturbed problem \eqref{eq:perturbed_hqcqp} with the $P$ defined above.
%    \begin{mini*}
%        {}{\trans{x}\left(Q^0 + \varepsilon_tP\right)x}{}{}
%        \addConstraint{\trans{x}Q^px}{\leq b_p, \quad}{p = 1, \ldots, m.}
%    \end{mini*}
    The corresponding system \eqref{eq:system} is
    \begin{equation*}
        y \geq 0, \quad
        S(y) = Q^0 + \varepsilon_tP + \sum_{p = 1}^m y_pQ^p \succeq O, \quad
        [S(y)]_{k\ell} = [Q^0]_{k\ell} + \varepsilon_tp_{k\ell} + \sum_{p = 1}^m y_p[Q^p]_{k\ell} = 0.
    \end{equation*}
    Now we analyze the feasibility of the system for any $(k, \ell) \in \ioffdiagonal{n}$. 
    \begin{enumerate}[label={(\alph*)}]
        \item
            If $[Q^0]_{k\ell} \neq 0$, then % it follows that
            \begin{equation*}
                [Q^0]_{k\ell} + \sum_{p = 1}^m y_p[Q^p]_{k\ell} = \begin{cases}
                    > 0 & \text{if $[Q^0]_{k\ell} > 0$,} \\
                    < 0 & \text{if $[Q^0]_{k\ell} < 0$,}
                \end{cases}
            \end{equation*}
            by the sign-definite assumption on the set $\{[Q^0]_{k\ell}, [Q^1]_{k\ell}, \ldots, [Q^m]_{k\ell}\}$.
            We have $[S(y)]_{k\ell} \neq 0$ for any $y \geq 0$ from $p_{k\ell} = 0$,
            therefore the system has no solution.
        \item
            If $[Q^0]_{k\ell} = 0$,
            then for any $y \geq 0$,
            \begin{equation*}
                [Q^0]_{k\ell} + \sum_{p = 1}^m y_p[Q^p]_{k\ell} = \begin{cases}
                    \geq 0 & \text{if $\sum_{p = 1}^m [Q^p]_{k\ell} \geq 0$,} \\
                    \leq 0 & \text{otherwise.}
                \end{cases}
            \end{equation*}
            From
            \begin{equation*}
                \varepsilon_tp_{k\ell} = \begin{cases}
                    \varepsilon > 0 & \text{if $\sum_{p = 1}^m [Q^p]_{k\ell} \geq 0$,} \\
                    -\varepsilon < 0 & \text{otherwise,}
                \end{cases}
            \end{equation*}
            we have $[S(y)]_{k\ell} \neq 0$,
            which  implies that the system has no solution.
    \end{enumerate}
    As  (a) and (b) cover all possible cases,
    the SDP relaxation of the $\varepsilon_t$-perturbed problem is exact
    for any $\varepsilon_t$ by \autoref{thm:main},
    hence  the original QCQP \eqref{eq:hqcqp} is also exact by \autoref{lem:perturbation_exactness}.
\end{proof} % }}}
\noindent
%This proposition shows one of easy-to-use conditions
%which does not depend on the system.
%On the contrary,
%a set $\{q^0_{k\ell}, q^1_{k\ell}, \ldots, q^m_{k\ell}\}$ can be
%easily created without additional calculations.
We will use Corollary \ref{prop:exact_condition_for_signdefinite}
to prove the exact  SDP relaxation for a class of QCQPs in \autoref{ssec:gtrs}.

Next, we examine  the following QCQP:
\begin{mini} 
    {}{\trans{x}Q^0x}{\label{eq:hqcqp_one_equality}}{}
    \addConstraint{\trans{x}Q^1x}{= b_p.}
\end{mini}
%
%At first glance, the problem has only one constraint;
%however, 
To write \eqref{eq:hqcqp_one_equality} in the form of \eqref{eq:hqcqp}, the equality constraint is converted  into
two inequality constraints $\trans{x}Q^1x \leq b_1$ and $\trans{x}(-Q^1)x \leq -b_1$.
It is clear that  problem \eqref{eq:hqcqp_one_equality} is a special case of QCQP \eqref{eq:hqcqp} with two constraints.
Since both matrices $Q^1$ and $-Q^1$ appear in
the inequality constraints of \eqref{eq:hqcqp_one_equality},
the set $\{[Q^0]_{ij}, [Q^1]_{ij}, -[Q^1]_{ij}\}$ is not sign-definite
unless $[Q^1]_{ij} =0$.
As a result,  \autoref{prop:exact_condition_for_signdefinite} cannot be used to determine the exactness of 
%we cannot check whether 
the SDP relaxation of \eqref{eq:hqcqp}. % is exact
%by \autoref{prop:exact_condition_for_signdefinite}.
For \eqref{eq:hqcqp_one_equality}, we propose  the following exactness condition.
%only for \eqref{eq:hqcqp_one_equality}.
%
\begin{corollary} \label{prop:exact_condition_for_one_equality}
    Suppose \eqref{eq:hqcqp} is a forest-structured QCQP.
    If for all $(i, j) \in \ioffdiagonal{n}$ such that $[Q^1]_{ij} \neq 0$,
    \begin{equation*}
        Q^0 - \frac{[Q^0]_{ij}}{[Q^1]_{ij}} Q^1 \not\succeq O,
    \end{equation*}
    then the SDP relaxation for \eqref{eq:hqcqp_one_equality} is exact.
\end{corollary}
\begin{proof} % {{{
    Let $\EC$ be the aggregate sparsity structure of the SDP relaxation
    defined in \autoref{ssec:aggregate_sparsity}.
    If $(\FC_{{k}{\ell}})$ are infeasible for all $(k, \ell) \in \ioffdiagonal{n}$,
    then we know from \autoref{thm:main} that
    the SDP relaxation is exact.
    Thus, we analyze the feasibility of the system $(\FC_{{k}{\ell}})$ for
    each $(k, \ell) \in \ioffdiagonal{n}$.
    We compute $S(y)$ in the system as:
    \begin{align*}
        S(y)
        &= Q^0 + y_1Q^1 + y_2(-Q^1),
            \quad y = \begin{bmatrix} y_1 \\ y_2 \end{bmatrix} \geq 0,\\
        &= Q^0 + z Q^1, \quad z \coloneqq y_1 - y_2 \in \bR.
    \end{align*}
    Then, the system $(\FC_{{k}{\ell}})$ is equivalent to the following system:
    \begin{equation} \label{eq:system_z}
        Q^0 + zQ^1 \succeq O, \quad
        [Q^0]_{k\ell} + z[Q^1]_{k\ell} = 0.
    \end{equation}

    \begin{enumerate}[label={(\alph*)}]
        \item
          %  \modify{If $(k, \ell) \not\in E(Q^1)$, then  $(k, \ell) \in E(Q^0)$,
          %  which implies $[Q^0]_{k\ell} \neq 0$ and $[Q^1]_{k\ell} = 0$.}{%
            If $[Q^1]_{k\ell} = 0$, then $[Q^0]_{k\ell} \neq 0$ by $(k, \ell) \in \ioffdiagonal{n}$.
            %}
            Since any $z \in \bR$ does not satisfy
            the second equation in \eqref{eq:system_z},
            the system \eqref{eq:system_z} is infeasible.
        \item
           % \modify{For $(k, \ell) \in E(Q^1)$,}{%
                If $[Q^1]_{k\ell} \neq 0$,
            %}
            we assume %, for the sake of contradiction,
            that a solution $z^*$ to \eqref{eq:system_z} exists.
            By solving the second equation in \eqref{eq:system_z} for $z^*$
            and substituting it into the first equation,
            we have
            \begin{equation*}
                Q^0 - \frac{[Q^0]_{ij}}{[Q^1]_{ij}} Q^1 \succeq O,
            \end{equation*}
            which is a contradiction. Thus, the desired result follows.
    \end{enumerate}
\end{proof} % }}}

 Burer and Ye~\cite{Burer2019} 
proposed several methods to extend their result of diagonal QCQPs to a general class of  QCQPs
with exact SDP relaxations.
However, for some QCQPs, it is difficult to show the exactness of  relaxations by their conditions;
for example, it is hard to prove the existence of an exact relaxation for the generalized trust-region subproblem
(see \autoref{sec:application_nontri}).

%\input sect4.tex
%!TEX root = ./main.tex
% \section{Application to non-tridiagonal QCQPs} 
\section{ Extension to a wider class of QCQPs based on simultaneous tridiagonalization} 
\label{sec:application_nontri}

To apply the results in  \autoref{sec:exactness_of_triqcqp} to a wider class of QCQPs,
we consider QCQPs
where the data matrices in \eqref{eq:qcqp} are not forest-structured.
 In this section, we suppose \autoref{asm:assumptions} \ref{asm:assumption_1} -- \ref{asm:assumption_3} and
that all the matrices of a QCQP are simultaneously tridiagonalizable.
In section 4.1, simultaneous tridiagonalization is discussed
in detail when $m=1$ (only one quadratic constraint). 
Then, a method to determine the exactness  of the SDP relaxation for these QCQPs is described.
In addition, we provide an alternative proof for the exactness of
 Generalized Trust-Region Subproblems (GTRS). 

The matrices $Q^0, Q^1, \ldots, Q^m \in \SMAT^n$ are called  simultaneous tridiagonalizable
if there exist a nonsingular matrix $U$ such that
$\trans{U}Q^0U,\trans{U}Q^1U, \ldots, \trans{U}Q^mU$ are tridiagonal matrices.
The simultaneous tridiagonalization is
a generalization of the simultaneous diagonalization  used in \cite{Burer2019}.
By replacing $U^{-1}x$ with  $\hat{x}$,
we obtain the tridiagonal QCQP equivalent to \eqref{eq:qcqp}:
\begin{mini}
    {}{\trans{\hat{x}}\left(\trans{U}Q^0U\right)\hat{x} + 2\trans{\left(\trans{U}q_0\right)}\hat{x} } 
        {\label{eq:qcqp_tridiagonalize}}{}
    \addConstraint{\trans{\hat{x} }\left(\trans{U}Q^pU\right)\hat{x}  + 2\trans{\left(\trans{U}q_p\right)}\hat{x} }
        {\leq b_p, \quad}{p = 1, \ldots, m.}
\end{mini}
The standard SDP relaxation of \eqref{eq:qcqp_tridiagonalize} is
 \begin{mini}
    {}{\innerp{\left(\trans{U}Q^0U\right)}{\widehat{X}} + 2\trans{\left(\trans{U}q_0\right)}\hat{x}}{\label{eq:Tsdr}}{} 
    \addConstraint{\innerp{\left(\trans{U}Q^pU\right)}{\widehat{X}} + 2\trans{\left(\trans{U}q_p\right)}\hat{x}} {\leq b_p, \quad}{p = 1, \ldots, m}
    \addConstraint{ \widehat{ X}}{\succeq \hat{x}\trans{\hat{x}}.}
\end{mini}
Obviously, if $\hat{x}$ is an optimal solution of \eqref{eq:qcqp_tridiagonalize}, then
$U\hat{x}$ is an optimal solution of the original problem~\eqref{eq:qcqp}.
We notice that the SDP relaxation   
\eqref{eq:sdr} of \eqref{eq:qcqp} is
at least as strong as the corresponding SDP relaxation \eqref{eq:Tsdr} for \eqref{eq:qcqp_tridiagonalize}. As a result,
if \eqref{eq:qcqp_tridiagonalize} has an exact relaxation,
then \eqref{eq:qcqp} also has an exact relaxation.
When \eqref{eq:qcqp_tridiagonalize} becomes a forest-structured QCQP in the homogeneous form,
 the exactness conditions in \autoref{sec:exactness_of_triqcqp}
can be applied to \eqref{eq:qcqp_tridiagonalize}.
We note that the exactness of the SDP relaxation for  QCQPs \eqref{eq:qcqp} can be determined
if their data matrices are simultaneous tridiagonalizable even when they are not forest-structured.

%--------------------------------------
\subsection{Simultaneous tridiagonalization} 
\label{sec:simul_tri}
Simultaneous tridiagonalization  of multiple matrices
is an extension of  simultaneous diagonalization,
and it can be achieved by finding 
  a nonsingular matrix
that transforms all matrices to tridiagonal matrices.
\noindent
Recently, Sidje~\cite{Sidje2011} (See also Garvey et al.~\cite{Garvey2003}) introduced conditions
under which two matrices are simultaneous tridiagonalizable.
\begin{proposition} ~\cite{Sidje2011}  \label{prop:tridiagonalizable}
    Let $K, M \in \bS^n$ and $0 \neq \gamma \in \bR$.
    Suppose that the matrix pencil $K - \gamma M$ is nonsingular.
    Then,  $K$ and $M$ are simultaneously tridiagonalizable.
\end{proposition}
\noindent
If the data matrices are simultaneously tridiagonalizable with the nonsingular matrix $U$,
we also need a method to compute $U$.
Sidje proposed a method to compute such a nonsingular matrix $U$
on the basis of Householder reflections.
We
briefly describe his method for the simultaneous tridiagonalization,
and then analyze this method to find new properties which will be used in the proof of the GTRS. 

%We first outline an overview of
In the beginning of the Sidje's recursive procedure~\cite{Sidje2011},
 the matrices are initialized as  $K^n \coloneqq K$, $M^n \coloneqq M \in \bS^n$.
Then, an appropriate nonsingular matrix
$U^k = \left[1, \trans{0_{k - 1}}; u_k, \widetilde{U}^k\right] \in \bR^{k \times k}$
with nonsingular $\widetilde{U}^k \in \bR^{(k - 1) \times (k - 1)}$ and
$u_k \in \bR^{k - 1}$ is chosen
such that
\begin{equation} \label{eq:tridiagonalization_one_recursion}
    \trans{(U^k)}K^kU^k = \left[\begin{array}{c:c}
        \xi_k & \begin{matrix} \tau_k & \trans{0_{k - 2}} \end{matrix} \\ \hdashline
        \begin{matrix} \tau_k \\ 0_{k - 2} \end{matrix} & K^{k - 1}
    \end{array}\right], \quad
    \trans{(U^k)}M^kU^k = \left[\begin{array}{c:c}
        \nu_k & \begin{matrix} \sigma_k & \trans{0_{k - 2}} \end{matrix} \\ \hdashline
        \begin{matrix} \sigma_k \\ 0_{k - 2} \end{matrix} & M^{k - 1}
    \end{array}\right]
\end{equation}
at each step $k = n, \ldots, 2$. Here 
$\xi_k, \tau_k, \nu_k, \sigma_k \in \bR$, 
and $K^{k-1}, M^{k-1} \in \bR^{(k-1) \times (k-1)}$.
This procedure generates two tridiagonal matrices:
\begin{gather*} % {{{
    \trans{U}K^nU = \left[
        \begin{array}{cccc:c}
            \xi_n & \tau_n & & & \\
            \tau_n & \xi_{n - 1} & \ddots & & \\
            & \ddots & \ddots & \tau_3 & \\
            & & \tau_3 & \xi_2 & \tau_2 \\ \hdashline
            & & & \tau_2 & K^1
        \end{array}
    \right]\quad \textrm{and} \quad
    \trans{U}M^nU = \left[
        \begin{array}{cccc:c}
            \nu_n & \sigma_n & & & \\
            \sigma_n & \nu_{n - 1} & \ddots & \\
            & \ddots & \ddots & \sigma_3 & \\
            & & \sigma_3 & \nu_2 & \sigma_2 \\ \hdashline
            & & & \sigma_2 & M^1
        \end{array}
    \right],
\end{gather*} % }}}
where
\begin{equation*}
    U := U^n
    \left[\begin{array}{c:c} I_1 & \\ \hdashline & U^{n - 1} \end{array}\right] \cdots
    \left[\begin{array}{c:c} I_{n - 2} & \\ \hdashline & U^2 \end{array}\right].
\end{equation*}
Now, consider the step $k$.
To have nonzero elements only on the diagonals, superdiagonals, and subdiagonals
by the operation shown in \eqref{eq:tridiagonalization_one_recursion},
 $\widetilde{U}^k$ should satisfy  the following equation: % needs to hold the following equation:
\begin{equation} \label{eq:relationship_of_tau_and_sigma}
    \trans{\left(\widetilde{U}^k\right)}\left(K^k_{\{1\},\{2, \ldots, k\}} + K^k_{\{2,\ldots,k\}}u_k\right) = \tau_ke_1, \quad
    \trans{\left(\widetilde{U}^k\right)}\left(M^k_{\{1\},\{2, \ldots, k\}} + M^k_{\{2,\ldots,k\}}u_k\right) = \sigma_ke_1,
\end{equation}
where $\e_1\in\Real^{k-1}$ and $K^{k}_{\{2,\ldots,k\}}$ means the submatrix of  $K^k$ obtained by removing the first row and column from $K^k$, as mentioned in section 2.1.
In his procedure, $\widetilde{U}^k$ is chosen to be a Householder reflection,
and therefore, nonsingular, and the existence of $u_k$ follows imposing
\begin{equation} \label{eq:linearly_dependence_by_u_gamma}
    \trans{\left(\widetilde{U}^k\right)}\left(K^k_{\{1\},\{2, \ldots, k\}} + K^k_{\{2,\ldots,k\}}u_k\right)
    = \gamma \trans{\left(\widetilde{U}^k\right)}\left(M^k_{\{1\},\{2, \ldots, k\}} + M^k_{\{2,\ldots,k\}}u_k\right)
\end{equation}
for $0\neq\gamma\in\bR$ since $K-\gamma M$ is nonsingular.
By restricting $\gamma$ to be positive,
we extend \autoref{prop:tridiagonalizable}. 
\begin{lemma} \label{lem:tridiagonalizable_with_sign_definite}
    Let $K, M \in \bS^n$ and $\bR\ni\gamma > 0$.
    Suppose that the matrix pencil $K - \gamma M$ is nonsingular.
    Then, there exists a nonsingular matrix $U \in \bS^n$
    that simultaneously tridiagonalizes $K$ and $M$.
    Moreover, for any $i \in [n - 1]$,  the $(i, i+1)$-th elements of
    $\trans{U}KU$ and $\trans{U}MU$ are sign-definite.
\end{lemma}
\begin{proof}
    Substituting \eqref{eq:relationship_of_tau_and_sigma}
    into \eqref{eq:linearly_dependence_by_u_gamma},
    we obtain $\tau_k = \gamma\sigma_k$.
    Since $\gamma > 0$,
    the set $\{\tau_k, \sigma_k\}$ is sign-definite for any $k = n, \ldots, 2$.
    The variables $\tau_k$ and $\sigma_k$ appear on
    the $(1, 2), \ldots, (n - 1, n)$-th elements
    of $\trans{U}K^nU$ and $\trans{U}M^nU$, respectively.
\end{proof}
\noindent
If  the elements on superdiagonal of all matrices can be transformed to sign-definite elements, then 
we can apply \autoref{prop:exact_condition_for_signdefinite}
to show  the exactness of the SDP relaxation, which is discussed next in \autoref{ssec:gtrs}.

%--------------------------------------------------------------------------

\subsection{Generalized trust-region subproblem} % {{{
\label{ssec:gtrs}
Consider the Generalized Trust-Region Subproblem (GTRS):
\begin{mini}
    {}{\trans{x}Q^0x + 2\trans{q_0}x}{\label{eq:trs}}{}
    \addConstraint{\trans{x}Q^1x + 2\trans{q_1}x}{\leq b_1.}
\end{mini}
The GTRS can be considered a generalization of the classical TRS
that minimizes a quadratic objective over an Euclidean ball,
i.e., the GTRS with $Q^1 \succ O$.
%In fact, we can easily check it
This fact can be seen by substituting $\sqrt{Q^1}x$ as a new variable $\tilde{x}$,
where  $\sqrt{Q^1}$ denotes the Cholesky factor of $Q^1$, i.e.,  $Q^1 = \trans{\sqrt{Q^1}}\sqrt{Q^1}$.
In the TRS, $Q^0$ is not necessarily   positive definite.
Although the TRS is nonlinear and nonconvex,
its SDP relaxation is well-known to be exact. 
The GTRS shares nice properties with the TRS.
For example, by using S-lemma,
it is proved that the SDP relaxation of the GTRS is always exact
under the Slater's condition. %~\cite{Polik2007}.

Since
the GTRS is clearly a QCQP \eqref{eq:qcqp} with only one constraint ($m = 1$),
we can also prove that
the GTRS admits an exact SDP relaxation
with the exactness conditions presented  in \autoref{sec:exactness_of_triqcqp}.
Although the result is certainly not new,
our proof shows a procedure on how to apply
 the exactness conditions for tridiagonal QCQPs to wider classes of QCQPs.
In fact,  the proof demonstrates how
to determine the exactness of a given QCQP in practice,
and it can be used to analyze  the exactness conditions for broader classes of QCQPs.

Any QCQP can be formulated in the  equivalent homogeneous QCQP as in \eqref{eq:hqcqp}, 
thus, it is sufficient to consider the following QCQP with an additional variable
 to discuss the exactness for the GTRS \eqref{eq:trs}:
\begin{mini}
    {}{\trans{x}\bar{Q^0}x}{\label{eq:htrs}}{}
    \addConstraint{\trans{x}\bar{Q^1}x}{\leq 0}
    \addConstraint{\trans{x}E_{11}x}{= 1,}
\end{mini}
where 
\begin{equation*}
    \bar{Q^p} \coloneqq \begin{bmatrix} -b_p & \trans{q_p} \\ q_p & Q^p \end{bmatrix} \quad (p=0,1)
    \quad \textrm{and} \quad b_{0} = 0.
\end{equation*}
For simplicity, we assume that the number of variables in \eqref{eq:htrs} is $n$. 
As \eqref{eq:htrs} has the additional equality constraint,
the problem \eqref{eq:htrs} is a QCQP with three inequality constraints,
and \eqref{eq:htrs} is no longer a GTRS.
In the subsequent discussion,
we describe the exactness for \eqref{eq:htrs} using the simultaneous tridiagonalization.
In particular, we show that the SDP relaxation of GTRS \eqref{eq:trs} is exact as follows.

\begin{theorem}
    Suppose that GTRS \eqref{eq:htrs} satisfies \autoref{asm:assumptions}.
    Then, the SDP relaxation of \eqref{eq:htrs} is exact.
\end{theorem}
\begin{proof} % {{{
    Let us first consider the case
    when $\bar{Q^0} - \gamma \bar{Q^1}$ is nonsingular for some $\gamma > 0$.
    By \autoref{lem:tridiagonalizable_with_sign_definite},
    we obtain a nonsingular matrix $U \in \bS^n$
    that simultaneously tridiagonalizes $Q^0$ and $Q^1$,
    and the $(i, i+1)$-th elements of
    $\trans{U}\bar{Q^0}U$ and $\trans{U}\bar{Q^1}U$ become sign-definite for any $i \in [n - 1]$.
    Call these tridiagonal matrices 
    \begin{equation*}
        R^p = [r^p_{ij}] \coloneqq \trans{U}\bar{Q^p}U,
        \quad p = 0, 1.
    \end{equation*}
    For any $i \in [n - 1]$, the set $\{r^0_{i,i+1}, r^1_{i,i+1}\}$ is sign-definite.
    By letting $y = U^{-1} x$,
    the homogeneous TRS \eqref{eq:htrs} can be transformed to an equivalent tridiagonal QCQP:
    \begin{mini}
        {}{\trans{y}R^0y}{\label{eq:htrs_transformed}}{}
        \addConstraint{\trans{y}R^1y}{\leq 0}
        \addConstraint{\trans{y}E_{11}y}{= 1.}
    \end{mini}
    Notice that the first row of $U$ is $[1 \ \trans{0_{n-1}}]$ by  the construction of section  \ref{sec:simul_tri}.
    By \autoref{prop:exact_condition_for_signdefinite},
    the SDP relaxation of \eqref{eq:htrs_transformed} is exact.
   As $U$ is nonsingular,
    the SDP relaxation of the original problem \eqref{eq:htrs} is also exact.

    Now consider the other case, i.e., there is no $\gamma > 0$
    such that $\bar{Q^0} - \gamma \bar{Q^1}$ is nonsingular.
    We will show that, for a fixed $\gamma > 0$ and any $\varepsilon > 0$,
    the SDP relaxation of the following $\varepsilon$-perturbed problem is exact:
    \begin{mini}
        {}{\trans{x}\left(\bar{Q^0} + \varepsilon I_n\right)x}{\label{eq:htrs_perturbed}}{}
        \addConstraint{\trans{x}\bar{Q^1}x}{\leq 0}
        \addConstraint{\trans{x}E_{11}x}{= 1.}
    \end{mini}
    Then, by \autoref{lem:perturbation_exactness},
    the SDP relaxation of the original problem \eqref{eq:htrs} is also exact.
    Since $\det\left(\bar{Q^0} - \gamma \bar{Q^1}\right) = 0$,
    $\bar{Q^0} - \gamma \bar{Q^1}$ can be diagonalized, as below:
    \begin{equation}
        \begin{bmatrix} \Lambda & \\ & O_{n - \mathrm{rk}} \end{bmatrix}  \label{DIAG}
        = \trans{P}\left(\bar{Q^0} - \gamma \bar{Q^1}\right)P,
    \end{equation}
    where $\mathrm{rk} \coloneqq \rank\left(\bar{Q^0} - \gamma \bar{Q^1}\right)$,
    $P \in \bR^{n \times n}$ is an orthogonal matrix,
    and $\Lambda \in \bS^\mathrm{rk}$ is a diagonal matrix.
    By adding  perturbation with sufficiently small  $\varepsilon > 0$
    to both sides of \eqref{DIAG}, % above equality,
    \begin{equation*}
        \begin{bmatrix} \Lambda & \\ & O_{n - \mathrm{rk}} \end{bmatrix} + \varepsilon I_n
        = \trans{P}\left(\bar{Q^0} - \gamma \bar{Q^1}\right)P + \varepsilon \trans{P}I_nP
        = \trans{P}\left(\bar{Q^0} + \varepsilon I_n - \gamma \bar{Q^1}\right)P,
    \end{equation*}
     $\left(\bar{Q^0} + \varepsilon I_n\right) - \gamma \bar{Q^1}$ becomes nonsingular.
    From the first case of this proof,
    the SDP relaxation of \eqref{eq:htrs_perturbed} must be exact.
\end{proof} % }}}
% }}}
% }}}

%\input sect5.tex
%!TEX root = ./main.tex
\section{Concluding remarks} 

We have presented sufficient conditions for the SDP relaxation of a class of QCQPs to be exact by
investigating the rank of the aggregated sparsity matrix of QCQPs. 
The class of QCQPs is forest-structured QCQPs that include tridiagonal QCQPs, arrow-type QCQPs,
and QCQPs with simultaneously tridiagonalizable data matrices.
The signs of  data matrix elements have not been used to determining the exactness of the SDP relaxation.

In our proof for the main results, we have utilized the fact that
any symmetric positive semidefinite matrix whose graph is a tree has rank at least $n-1$.  It is shown in \cite{JOHNSON06} that
 for any non-tree the minimum rank of a positive semidefinite matrix is less than $n-1$.

We have also extended our results to non-tridiagonal QCQPs by improving a computing method  proposed in \cite{Sidje2011} for simultaneous tridiagonalization.
The exactness of the SDP relaxation of the GTRS can be proved by the simultaneous tridiagonalization
and our results in section 3.

For a future work, we want to extend the result to a wider class of QCQPs.
Algorithms for simultaneous tridiagonalization of multiple matrices will be further studied to apply our results. 

% \bibliographystyle{plain} %nat-reversed} %elsarticle-harv}
%\bibliographystyle{abbrv} %nat-reversed} %elsarticle-harv}

%\begin{small}
%\bibliography{./reference}
%\input{appendix.tex}
%\end{small}

\end{document}